\documentclass[preprint,12pt]{elsarticle}

\usepackage{array}
\usepackage{bigstrut}
\usepackage{graphicx}
\usepackage{epsfig}
\usepackage{amssymb}
\usepackage[intlimits]{amsmath}
\usepackage{amsmath}
\usepackage{slashbox}
\usepackage{multirow}
\usepackage{hhline}	
\usepackage{makecell}
\usepackage{float}

\newcommand{\q}{\quad}
\newcommand{\ee}{{\rm e}\hspace{1pt}}
\newcommand{\dd}{\hspace{0.5pt}{\rm d}\hspace{0.5pt}}

 \usepackage{amsthm}
\usepackage{lineno}
\usepackage{amssymb,latexsym,amscd,amsmath,amsfonts,enumerate,supertabular}
\usepackage{tabularx}

\journal{}

\numberwithin{equation}{section}
\newtheorem{lemma}[]{Lemma}
\numberwithin{lemma}{section}
\newtheorem{theorem}[]{Theorem}
\numberwithin{theorem}{section}

\newtheorem{example}[]{Example}
\numberwithin{example}{section}
\newcommand{\overbar}[1]{\mkern 1.5mu\overline{\mkern-1.5mu#1\mkern-1.5mu}\mkern 1.5mu}

    \makeatletter
    \def\ps@pprintTitle{%
      \let\@oddhead\@empty
      \let\@evenhead\@empty
      \let\@oddfoot\@empty
      \let\@evenfoot\@oddfoot
    }
    \makeatother

\begin{document}

\begin{frontmatter}

\title{{\bf  Fourth-order two-stage explicit exponential integrators for solving differential equations }}
\author{Vu Thai Luan }
\ead{vluan@ucmerced.edu} 
\address{School of Natural Sciences, University of California, Merced, 5200 North Lake Road,\\
Merced, CA 95343, USA}

\begin{abstract}
\small  
\indent Among the family of fourth-order time integration schemes, the two-stage Gauss--Legendre method, which is an implicit Runge--Kutta method based on collocation, is the only superconvergent. The computational cost of this implicit scheme for large systems, however, is very high since it requires solving a nonlinear system at every step.  Surprisingly, in this work we show that one can construct and prove convergence results for exponential methods of order four which use two stages only. Specifically, we derive two new fourth-order two-stage exponential Rosenbrock schemes for solving large systems of differential equations. Moreover, since the newly schemes are not only superconvergent but also fully explicit, they clearly offer great advantages over the two-stage Gauss--Legendre method as well as other time integration schemes. Numerical experiments are given to demonstrate the efficiency of the new integrators.   
 
\small  
\end{abstract}


\end{frontmatter}
\section{Introduction}
Finding numerical solutions to time-dependent PDEs usually requires the time discretization of large systems of differential equations, which can be cast in the form
\begin{equation} \label{eq1.1}
u'(t)=F(u(t)), \q u(t_0)=u_0.
\end{equation}
Along with the development of numerical analysis, many methods have been designed for solving \eqref{eq1.1} numerically. Depending on the characteristics of each problem, one has to choose the right method. Nonstiff problems are usually integrated by using classical techniques such as explicit Runge--Kutta methods, multistep methods, and general linear methods (see \cite{HNW93}). The situation for stiff problems, however, is much more involved due to the fact that the Jacobian matrix often has a large norm or is even an unbounded operator. In this case, explicit methods have to face with stability issues. They are usually lack stability and are required to use extremely small time steps. To avoid this issue, various kinds of implicit methods have been proposed such as implicit Runge–Kutta methods (e.g., Gauss, Radau IA and IIA), BDF methods, Rosenbrock-type methods, just to name a few. For details of such methods we refer the reader to the excellent monograph \cite{HW96}. The downside of implicit methods, however, is their computational costs that are usually too high caused by solving large nonlinear system at every step. In order to overcome the two mentioned issues of such classical explicit and implicit methods, exponential integrators has been introduced (see the review paper \cite{HO10} for details). This field has grown significantly since 1998 and it has been shown that the integrators are highly competitive, see for example  \cite{HLS98, CM02,  HO05, Tok06, LO12b, LO14a, LO14b, PA14, rainwater14, LO16, LMG16}. High-order exponential integrators for stiff problems have been proposed in \cite{L13}.

In this work, we pay attention to a very completive and good candidate for solving stiff problems, the so-called exponential Rosenbrock methods, which is first proposed in \cite{HO06}. The idea  is, first to make a continuous linearization of the vector field $F$ along the numerical solution $u_n$  of \eqref{eq1.1} (due to Pope \cite{P63}) leading to semilinear problems
\begin{equation} \label{eq1.2}
u'(t)=J_n u(t)+ g_n(u(t))
\end{equation}
with the Jacobian $J_n$ and the nonlinearity $g_n$ are
\begin{equation} \label{eq1.3}
J_n=\frac{\partial F}{\partial u}(u_n), \q  g_n (u)=F(u)-J_{n} u,
\end{equation}
and then to apply exponential Runge--Kutta methods \cite{HO05} to \eqref{eq1.2} which resulted in exponential Rosenbrock methods. They have been studied intensively in a series of papers \cite{HO06, HOS09, LO14a, LO16}. Methods up to order 6 have been derived in \cite{LO16} and the stiff order conditions for methods up to arbitrary order are given in \cite{LO13}. One of the great advantages of  exponential Rosenbrock methods is that the Jacobian of the nonlinearities $g_n (u)$ vanishes at the numerical solution $u_n$ (see \eqref{eq2.5}). This improves the stability, simplifies the stiff order conditions, and thus allows one to construct high-order methods with a few stages only. For instance, we mention the class of 1-stage method of order 2 (considered as a superconvergence scheme), namely the exponential Rosenbrock-Euler method (see \cite{HOS09}) and the class of 3-stage method of order 5, $\mathtt{exprb53s3}$ (see \cite{LO14a}).

Our aim in this paper is to construct, analyze, and implement a class of fourth-order 2-stage explicit exponential Rosenbrock methods. 
This is motivated by the fact a 3-stage method  can get the maximum order $p=5$, see \cite{LO14a}. Moreover, it is not clear on the existence of a 2-stage method which has superconvergence property (order 4) both in the context of nonstiff and stiff problems.  By taking further investigate order conditions both in the classical and stiff sense, we will show that one can construct such methods.

The remainder of the paper is organized as follows. In Section~\ref{sc2}, we recall the exponential Rosenbrock schemes (including 2-stage methods) and present our motivation to this work. The construction of classical fourth-order 2-stage schemes is given  in Section~\ref{sc3}, where we give the classical order conditions (Lemma~\ref{lm2.1}), show the convergence result (Theorem~\ref{theorem2.1}), and derive the scheme $\mathtt{exprb42N}$ (see \eqref{eq2.20}). Inspired by these results, in Section~\ref{sc4} we show that, under the regularity assumptions on the problem, it is even possible to construct a stiffly accurate fourth-order 2-stage scheme. The main results of this section are Lemma~\ref{lm4.1} (stiff order conditions), Theorem~\ref{th4.1} (convergence), and the scheme $\mathtt{exprb42}$ (see \eqref{eq4.12}). Section~\ref{sc5} discusses variable stepsizes implementation for the two newly constructed integrators. Finally, in Section~\ref{sc6}  we verify the convergence results and show the efficiency of the two new integrators on a set of numerical examples.
\section{Numerical method and motivation}\label{sc2}
We start off by recalling the explicit exponential Rosenbrock-type methods for solving  \eqref{eq1.1}, see \cite{HOS09}:
\begin{subequations} \label{eq1.4}
\begin{align}
U_{ni}&= u_n + c_i h \varphi _{1} ( c_i h J_n)F(u_n) +h \sum_{j=2}^{i-1}a_{ij}(h J_n) D_{nj}, \label{eq1.4a} \\
u_{n+1}& = u_n + h \varphi _{1} ( h J_n)F(u_n) + h \sum_{i=2}^{s}b_{i}(h J_n) D_{ni}  \label{eq1.4b}
\end{align}
\end{subequations}
with
$
  D_{ni}= g_n ( U_{ni})- g_n(u_n ), \  2\leq i\leq s.
$
Here  $u_n \approx u(t_n)$, $c_i $ are the nodes, $s$ is the number of stages, $U_{ni}\approx u(t_n +c_i h)$, $J_n$ and $g_n$ are given in~\eqref{eq1.3}, $h = t_{n+1}-t_n >0$ denotes the time step. The coefficients $a_{ij}(z)$ and $b_i (z)$ are usually chosen as linear combinations of the corresponding entire functions $\varphi_{k}(c_i z)$ and $\varphi_{k}(z)$, where
\begin{equation} \label{eq1.5}
\varphi_0(z) = \ee^z,\qquad \varphi _{k}(z)=\int_{0}^{1} \ee^{(1-\theta )z} \frac{\theta ^{k-1}}{(k-1)!}\dd\theta , \quad k\geq 1.
\end{equation}
These functions satisfy the relation
\begin{equation} \label{eq1.6}
\varphi _{k}(z)=\frac{1}{k!}+z\varphi _{k+1}(z), \q k\geq 0.
\end{equation}

So far, it is known that the one-stage ($s=1$) second-order method, the so-called exponential Rosenbrock-Euler method
 \begin{equation} \label{eq1.7}
u_{n+1}= u_n + h \varphi _{1} ( h J_n)F(u_n),
\end{equation}
is the only superconvergent exponential integrator. 
It is shown in \cite{HOS09} that the 2-stage schemes, which read as
\begin{subequations} \label{eq2.1}
\begin{align}
 U_{n2}&= u_n + c_2 h \varphi _{1} ( c_2 h J_n)F(u_n), \label{eq2.1a} \\
u_{n+1}& = u_n + h \varphi _{1} ( h J_n)F(u_n) + h b_{2}(h J_n) (g_n (U_{n2})- g_n(u_n )),   \label{eq2.1b}
\end{align}
\end{subequations}
can attain third-order accuracy, see for example the scheme $\mathtt{exprb32}$ given in \cite{HOS09}. Obviously, from this one can easily derive a corresponding classical third-order scheme. However, the question whether or not a 2-stage method of order 4 exists is still open. On the other hand, in \cite{LO14a} it is shown that a 3-stage method can reach the maximum order $p=5$. Therefore, our aim in this work is to answer the question of superconvergence for the class of 2-stage methods.
It is thus important to further investigate order conditions for such a 2-stage scheme \eqref{eq2.1}, both in the classical and stiffly accurate situation. For this purpose, our idea is to analyze local errors directly as done in \cite{LO14a}. Namely, we will study one step integration scheme \eqref{eq2.1} with the initial values on the exact solution $\tilde{u}_n=u(t_n)$, i.e.
\begin{subequations} \label{eq2.2}
\begin{align}
\overbar{U}_{n2}&=\tilde{u}_n + c_2 h \varphi _{1} (c_2 h \tilde{J}_n)F(\tilde{u}_n), \label{eq2.2a}\\
\overbar{u}_{n+1}&= \tilde{u}_n + h \varphi _{1} (h \tilde{J}_n) F(\tilde{u}_n) + h b_{2}(h \tilde{J}_n) (\tilde{g}_n(\overbar{U}_{n2})-\tilde{g}_n(\tilde{u}_n). \label{eq2.2b}
\end{align}
\end{subequations}
Similarly to \eqref{eq1.3}, here
\begin{equation} \label{eq2.2add} 
 \tilde{J}_n=\dfrac{\partial F }{ \partial u} (\tilde{u}_n), \q \tilde{g}_n(u)=F(u)-\tilde{J}_n u, 
 \end{equation}
which are resulted from the linearization of \eqref{eq1.1} at $\tilde{u}_n$, i.e. 
\begin{equation} \label{eq2.3} 
u'(t)=\tilde{J}_n u(t)+ \tilde{g}_n(u(t)).
\end{equation}
Let
\begin{equation} \label{eq2.3add}
\tilde{e}_{n+1}=\bar{u}_{n+1}- \tilde{u}_{n+1}
\end{equation}
denote the local error, i.e., the error of the numerical solution after one step with initial value on the exact solution $\tilde{u}_{n}$. Since the structure of \eqref{eq2.2} allows to treat the linear part of \eqref{eq1.2} exactly (see Remark~1 below) and the fact that 
\begin{equation} \label{eq2.5} 
\frac{\partial \tilde{g}_n}{ \partial u}(\tilde{u}_n)=\frac{\partial }{ \partial u}(F(u)- \tilde{J}_n u)(\tilde{u}_n)=\tilde{J}_n-\tilde{J}_n=0,
\end{equation}
it is hoped that one can further simplify order conditions and derive from that the right coefficient $b_{2}(h \tilde{J}_n)$, which gives order of consistency 5 for  the local error, i.e. $\tilde{e}_{n+1}=\mathcal{O}(h^{5})$. 

For the remaining of the paper, we will focus on both cases: nonstiff and stiff problems. Our analysis will be performed in a Banach space $X$ with norm $\| \cdot \|$. 
 
\section{Construction of classical fourth-order 2-stage exponential Rosenbrock schemes}\label{sc3}
In this section we consider the case where the vector field $F(u)$ is a nonlinear function with a moderate Lipschitz constant. In other words, the problem \eqref{eq1.1} is supposed to be nonstiff.  We thus can make use of the following assumption.

{\em Assumption 1. Suppose that \eqref{eq1.1} possesses a sufficiently smooth solution $u: [0, T]\rightarrow X$, with derivatives in $X$ and that  $F: X \rightarrow X$ is sufficiently often Fr\'echet differentiable in a strip along the exact solution. All occurring derivatives are assumed to be bounded.} 

Clearly, under this assumption, $g_n(u)=F(u)-J_n u$ is also sufficiently often Fr\'echet differentiable as well as satisfies the Lipschitz condition in a strip along the exact solution.

We note for later use that under Assumption~1 one can expand $\varphi _{1} (c_2 h \tilde{J}_n)$ (by using the recurrence relation \eqref{eq1.6}) and $b_{2}(h \tilde{J}_n)$ appearing in \eqref{eq2.2} as
\begin{subequations} \label{eq2.4}
\begin{align}
\varphi _{1} (c_2 h \tilde{J}_n)&=\sum_{k\geq 0} \frac{(c_2 h \tilde{J}_n)^{k}}{(k+1)!}=I+\frac{1}{2!}c_2 h \tilde{J}_n + \frac{1}{3!}c^2_2 h^2 \tilde{J}^2_n +\mathcal{O}(h^3), \label{eq2.4a} \\
b_{2}(h \tilde{J}_n)&=\sum_{k\geq 0}\beta_k (h \tilde{J}_n)^{k}=\beta_0 I+ \beta_1 h \tilde{J}_n  +\mathcal{O}(h^2) . \label{eq2.4b}
\end{align}
\end{subequations}
We now derive an expansion of the numerical solution $\overbar{u}_{n+1}$.
\subsection{Expansion of the numerical solution}
Let  $\tilde{u}'_n, \tilde{u}''_n, \tilde{u}'''_n$ denote the first, second, and third derivative of the exact solution  $u(t)$ of \eqref{eq1.1} 
evaluated at time $t_n$, respectively. We further denote $\dfrac{\partial \tilde{g}_n}{ \partial u}(u), \dfrac{\partial^2 \tilde{g}_n}{ \partial u^2}(u),$ and  $\dfrac{\partial^3 \tilde{g}_n}{ \partial u^3}(u)$  by $\tilde{g}'_n(u), \tilde{g}''_n(u),$ and $\tilde{g}^{(3)}_n(u)$, respectively. 

By using \eqref{eq2.5} 
and differentiating the equation \eqref{eq2.3} twice, we obtain  
\begin{equation} \label{eq2.6} 
\tilde{J}_n \tilde{u}'_n=\tilde{u}''_n, \q \tilde{J}^2_n \tilde{u}'_n=\tilde{u}'''_n-\tilde{g}''_n(\tilde{u}_n)(\tilde{u}'_n, \tilde{u}'_n).
\end{equation}
Inserting $F(\tilde{u}_n)= \tilde{u}'_n$ and \eqref{eq2.4a} into \eqref{eq2.2a} with the help of the identities in  \eqref{eq2.6} gives 
\begin{equation} \label{eq2.7}
\overbar{U}_{n2}=\tilde{u}_n+c_2 h \tilde{u}'_n+ \frac{1}{2!}c^2_2 h^2  \tilde{u}''_n+ \frac{1}{3!}c^3_2 h^3  \big(\tilde{u}'''_n-\tilde{g}''_n(\tilde{u}_n)(\tilde{u}'_n, \tilde{u}'_n)\big)+ \mathcal{O}(h^4).
\end{equation}
By employing \eqref{eq2.7} and \eqref{eq2.5}, we next expand  $\tilde{g}_n(\overbar{U}_{n2})$ in a Taylor series at $\tilde{u}_n$ to get
\begin{equation} \label{eq2.8}
\begin{aligned}
\tilde{g}_n(\overbar{U}_{n2})-\tilde{g}_n(\tilde{u}_n)=&\frac{1}{2!}c^2_2 h^2 \tilde{g}''_n(\tilde{u}_n)(\tilde{u}'_n, \tilde{u}'_n)+\frac{1}{3!}c^3_2 h^3  \big( \tilde{g}^{(3)}_n(\tilde{u}_n)(\tilde{u}'_n, \tilde{u}'_n, \tilde{u}'_n)\\
 +&3 \tilde{g}''_n(\tilde{u}_n)(\tilde{u}'_n, \tilde{u}''_n) \big) +\mathcal{O}(h^4). 
\end{aligned}
\end{equation}
Inserting \eqref{eq2.4b} and \eqref{eq2.8} into \eqref{eq2.2b} yields the following expansion 
\begin{equation} \label{eq2.9}
\begin{aligned}
\overbar{u}_{n+1}&= \tilde{u}_n + h \varphi _{1} (h \tilde{J}_n) F(\tilde{u}_n) +h^3\frac{1}{2!}\beta_0 c^2_2  \tilde{g}''_n(\tilde{u}_n)(\tilde{u}'_n, \tilde{u}'_n)\\
&+ h^4 \Big( \frac{1}{3!}\beta_0 c^3_2  \big( \tilde{g}^{(3)}_n(\tilde{u}_n)(\tilde{u}'_n, \tilde{u}'_n, \tilde{u}'_n)+3 \tilde{g}''_n(\tilde{u}_n)(\tilde{u}'_n, \tilde{u}''_n) \big) +\frac{1}{2!}\beta_1 c^2_2 \tilde{J}_n \tilde{g}''_n(\tilde{u}_n)(\tilde{u}'_n, \tilde{u}'_n) \Big)\\
& +\mathcal{O}(h^5). 
\end{aligned}
\end{equation}
\emph{Remark~1.} It can be seen from the expansion of the numerical solution in \eqref{eq2.9} that we do not expand the term $\varphi _{1} (h \tilde{J}_n)$ in a power series of $h \tilde{J}_n$ as done for $\varphi _{1} (c_2 h \tilde{J}_n)$  in \eqref{eq2.4a}. The reason for keeping that term is because the sum of the first two terms in \eqref{eq2.9} can be rewritten, by using the fact that $\varphi _{1}(z)=(e^z-1)/z$ and $F(\tilde{u}_n)=\tilde{J}_n \tilde{u}_n+\tilde{g}_n (\tilde{u}_n)$, as 
\begin{equation} \label{eq2.9add}
 \tilde{u}_n + h \varphi _{1} (h \tilde{J}_n) F(\tilde{u}_n)=\ee^{h \tilde{J}_{n}}\tilde{u}_n+h\varphi _{1} (h \tilde{J}_n)\tilde{g}_n (\tilde{u}_n) 
\end{equation}
 which can be  used to treat the linear part of \eqref{eq1.2} exactly as seen in the expansion of the exact solution as follows. 

\subsection{Expansion of the exact solution}
Expressing the exact solution of (\ref{eq1.2}) at time $t_{n+1}$ by the variation-of-constants formula gives
\begin{equation} \label{eq2.10}
\tilde{u}_{n+1}=u(t_{n+1})=\ee^{h \tilde{J}_{n}}\tilde{u}_n+h \int_{0}^{1} \ee^{(1-\theta )h \tilde{J}_{n}} \tilde{g}_n(u(t_n +\theta h))\dd\theta
\end{equation}
which can be rewritten as 
\begin{equation} \label{eq2.11}
\tilde{u}_{n+1}=\ee^{h \tilde{J}_{n}}\tilde{u}_n+h \int_{0}^{1} \ee^{(1-\theta )h \tilde{J}_{n}} \tilde{g}_n(\tilde{u}_n)\dd\theta+ h \int_{0}^{1} \ee^{(1-\theta )h \tilde{J}_{n}} \big( \tilde{g}_n(u(t_n +\theta h))-\tilde{g}_n(\tilde{u}_n) \big) \dd\theta.
\end{equation}
One can realize that the sum of the first two terms of \eqref{eq2.11} is exactly equal to \eqref{eq2.9add}. Next, by employing \eqref{eq2.5} we expand $\tilde{g}_n(u(t_n +\theta h)$ in a Taylor series at $\tilde{u}_n$ and insert the obtained results into the third term of  \eqref{eq2.11} as done in \cite[Sec. 3.2]{LO14a},  which finally gives
\begin{equation}  \label{eq2.12}
\begin{aligned}
\tilde{u}_{n+1}=\tilde{u}_n &+  h  \varphi _{1} ( h \tilde{J}_{n}) F(\tilde{u}_n)+h^{3} \varphi _{3} ( h \tilde{J}_{n})  \tilde{g}''_n(\tilde{u}_n)(\tilde{u}'_n,\tilde{u}'_n)\\
&+ h^{4}  \varphi _{4} ( h \tilde{J}_{n}) \big( \tilde{g}^{(3)}_n(\tilde{u}_n)(\tilde{u}'_n, \tilde{u}'_n, \tilde{u}'_n)+3 \tilde{g}''_n(\tilde{u}_n)(\tilde{u}'_n, \tilde{u}''_n) \big)+\mathcal{O}(h^{5}).
\end{aligned}
\end{equation}
We now insert the following expansions 
\begin{equation}  \label{eq2.13}
\begin{aligned}
\varphi_{3} ( h \tilde{J}_{n})&=\frac{1}{3!}I+\frac{1}{4!} h\tilde{J}_{n} +  (h\tilde{J}_{n})^2 \varphi_{5} ( h \tilde{J}_{n})\\
\varphi_{4} ( h \tilde{J}_{n})&=\frac{1}{4!}I+(h\tilde{J}_{n}) \varphi_{5} ( h \tilde{J}_{n})
\end{aligned}
\end{equation}
obtained by using \eqref{eq1.6} into \eqref{eq2.12} to get 
\begin{equation}  \label{eq2.14}
\begin{aligned}
\tilde{u}_{n+1}=\tilde{u}_n &+  h  \varphi _{1} ( h \tilde{J}_{n}) F(\tilde{u}_n)+h^{3} \frac{1}{3!} \tilde{g}''_n(\tilde{u}_n)(\tilde{u}'_n,\tilde{u}'_n)\\
&+ h^{4} \frac{1}{4!} \Big( \tilde{g}^{(3)}_n(\tilde{u}_n)(\tilde{u}'_n, \tilde{u}'_n, \tilde{u}'_n)+3 \tilde{g}''_n(\tilde{u}_n)(\tilde{u}'_n, \tilde{u}''_n)+\tilde{J}_{n}\tilde{g}''_n(\tilde{u}_n)(\tilde{u}'_n,\tilde{u}'_n)  \Big)+\mathcal{O}(h^{5}).
\end{aligned}
\end{equation}
With this expansion of the exact solution at hand, we are now ready to derive (classical) order conditions for 2-stage methods of order 4.
\subsection{Local error and order conditions for fourth-order 2-stage methods} 
By subtracting \eqref{eq2.14} from \eqref{eq2.9}, it is straightforward to derive the following result for the local error $\tilde{e}_{n+1}$.
\begin{lemma} \label{lm2.1}
Under Assumption 1, an explicit 2-stage exponential Rosenbrock scheme \eqref{eq2.1} in which the coefficient $b_2(h J_n)$ can be expanded as \eqref{eq2.4b}, has order of consistency five, i.e. the local error $\tilde{e}_{n+1}=\mathcal{O}(h^{5})$ if the following order conditions are fulfilled 
\begin{equation} \label{eq2.16}
\frac{\beta_0 c^2_2}{2!}=\frac{1}{3!}, \q \frac{\beta_0 c^3_2}{3!}=\frac{1}{4!}, \q \frac{\beta_1 c^2_2}{2!}=\frac{1}{4!},
\end{equation}
that is equivalent to  
\begin{equation} \label{eq2.17}
c_2=\frac{3}{4}, \q \beta_0=\frac{16}{27}, \q \beta_1=\frac{4}{27}.
\end{equation}
Here the remainder term of $\tilde{e}_{n+1}$, which is hidden behind the Landau notation $\mathcal{O}(\cdot)$, is bounded by $Ch^5$ with a constant $C$ that depends on $\| \tilde{J}_{n}\|$.
\qed
\end{lemma}
The result of Lemma~\ref{lm2.1} implies that the coefficient $b_2(h J_n)$ of \eqref{eq2.1} must satisfy the following expansion
\begin{equation} \label{eq2.18}
b_2 (h J_n)=\frac{16}{27}I +\frac{4}{27} h J_n + \mathcal{O}(h^2).
\end{equation}
in order for the method to have order of consistency five.
\subsection{Convergence result}
In the following we show that such a scheme \eqref{eq2.1} that takes $c_2=3/4$ and fulfills \eqref{eq2.18} is indeed convergent with a global error of order 4.
\begin{theorem}\label{theorem2.1}
Let the initial value problem \eqref{eq1.1} satisfies Assumption 1. Consider for its numerical solution  an explicit 2-stage  exponential Rosenbrock scheme \eqref{eq2.1} with $b_2(h J_n)$ satisfies \eqref{eq2.18} and the node $c_2=\frac{3}{4}$ (fulfilling the order conditions in \eqref{eq2.17}). Then, the method converges with order four, i.e.
 \begin{equation} \label{eq2.21}
 \|u_n-u(t_n)\|\leq C h^4
 \end{equation}
 on \ $t_0 \leq  t_n =t_0+nh \leq  T$ with a constant $C$ that depends on $n$ and $h$.
\end{theorem}
\begin{proof}
It is remaining to show that the numerical scheme \eqref{eq2.1} is stable. This is straightforward due to the fact that, under Assumption~1, the Jacobian $J(u)=\frac{\partial F (u)}{\partial u}$ also satisfies the Lipschitz condition in a strip along the exact solution $u$. Another possibility is to employ the stability condition of exponential Rosenbrock methods which is recalled in \eqref{eq4.13} in Section~\ref{subsection4.13} below. We thus omit the details.
\end{proof}
\subsection{Derivation of classical fourth-order 2-stage schemes}
Clearly, a 2-stage scheme \eqref{eq2.1} is derived if  the coefficient $b_2(h J_n)$ is identified. Since $b_2(h J_n)$ is usually chosen as linear combinations of some matrix functions $\varphi_k (h J_n)$, the condition \eqref{eq2.18} determines explicitly such a linear combination.
 For example, one can choose $b_2(h J_n)$ as a linear combination of $\varphi_1 (h J_n)$ and $\varphi_2 (h J_n)$ as $b_2(h J_n)=-\frac{8}{27}\varphi_1 (h J_n)+\frac{48}{27}\varphi_2 (h J_n) $ resulting in the following scheme which will be called $\mathtt{exprb42N}$:
\begin{subequations} \label{eq2.20}
\begin{align}
 U_{n2}&= u_n + \frac{3}{4} h \varphi _{1} ( \frac{3}{4} h J_n)F(u_n), \\
u_{n+1}& = u_n + h \varphi _{1} ( h J_n)F(u_n) + h \big(-\dfrac{8}{27}\varphi_1 (h J_n)+\dfrac{48}{27}\varphi_2 (h J_n) \big)(g_n (U_{n2})- g_n(u_n )).  
\end{align}
\end{subequations}
Note that  one can derive many other 2-stage fourth-order schemes like \eqref{eq2.20} as long as $b_2(h J_n)$ satisfies condition \eqref{eq2.18}.

\section{Construction of a stiff fourth-order 2-stage exponential Rosenbrock scheme}\label{sc4}
It should be mentioned that our convergence analysis presented in Section~\ref{sc3} cannot be applied if the Jacobian $J_n$ has a large norm or is even unbounded operator. The reason for that is simply because Assumption 1 and thus the expansions \eqref{eq2.4} and \eqref{eq2.13} are no longer valid. Unfortunately, this is usually the situation of stiff problems arising when discretizing the space dimension of many time dependent PDEs. Examples of such problems are diffusion-reaction equations, the heat equations, just to mention a few.  Therefore, in this section our aim is to design a 2-stage exponential Rosenbrock scheme of the form \eqref{eq2.1} that is superconvergent and works for such stiff problems. We will focus on the common case where the vector field $F(u)$ can be decomposed into two parts: the linear part which is stiff and the nonlinear part which is nonstiff, namely 
\begin{equation} \label{eq4.1}
u'(t)=F(u(t))=Au(t)+ g(u(t)), \q u(t_0)=u_0.
\end{equation}
In the subsequent analysis, we will use the  framework of strongly continuous semigroups in the Banach space $X$ (for instance, see \cite{EN2000,PAZY83}) to handle this type of stiff problems. In particular, throughout this section the following main assumptions (see also \cite{HOS09, LO14a}) will be employed.

{\em Assumption 2. The linear operator $A$ is the generator of a strongly continuous semigroup  $\ee^{tA}$ in $X$}.

{\em Assumption 3. Suppose that \eqref{eq4.1} possesses a sufficiently smooth solution $u: [0, T]\rightarrow X$, with derivatives in $X$ and that the nonlinearity $g: X \rightarrow X$ is sufficiently often Fr\'echet differentiable in a strip along the exact solution. All occurring derivatives are
supposed to be uniformly bounded}.

By using a standard perturbation result in \cite[Chap.~3.1]{PAZY83}, it is easy to infer from Assumptions 1 and 2 that the Jacobian 
\begin{equation} \label{eq4.2}
J=J(u)=\frac{\partial F}{\partial u}(u) = A + g'(u)
\end{equation}
also generates a strongly continuous semigroup. This implies that there exist constants $C$ and $\omega$ such that the bound
\begin{equation} \label{eq4.3}
\|\ee^{tJ}\|_{X\leftarrow X}\leq C\ee^{\omega t}, \quad t\geq 0
\end{equation}
holds uniformly in a neighborhood of the exact solution. As a consequence of the bound \eqref{eq4.3}, one can see that the coefficients  $\varphi_1 (h J_n), \varphi_1 (c_2 h J_n)$ and $b_{2}(h J_n)$ of the 2-stage exponential Rosenbrock scheme \eqref{eq2.1} are bounded operators. 
 Assumption 2 further implies that the Jacobian \eqref{eq4.2} and $g(u)$ are both locally Lipschitz in a strip
along the exact solution $u$. In particular, in a neighborhood of the exact solution we have
\begin{equation} \label{eq4.4}
\| J(u)-J(v)\|_{X\leftarrow X} =  \left\|  g'(u)- g'(v) \right\|_{X\leftarrow X} \leq L\|u-v\|.
\end{equation}
\subsection{Local error and relaxing stiff order conditions for 2-stage methods}
As we are interested in constructing a superconvergent 2-stage exponential Rosenbrock scheme \eqref{eq2.1} for solving \eqref{eq4.1}, one has to find the right coefficient $b_2(h J_n)$ which satisfies the \emph{stiff} order conditions for methods of order 4. In the following we will show that this can be done by using the new and simplified stiff order conditions for exponential Rosenbrock methods of order 4 given in \cite{LO14a, LO13} and relaxing one of them. For convenience, we display the local error expansion of 2-stage methods, which can be obtained at once by using the result of the local error for $s$-stage methods given in \cite[Sec.3.3]{LO14a}, as follows
\begin{equation} \label{eq4.5}
\begin{aligned}
\tilde{e}_{n+1}=& h^3 \big( b_2 (h \tilde{J}_n)\frac{c^2_2}{2!}-\varphi_3 (h \tilde{J}_n) \big)\tilde{g}''_n(\tilde{u}_n)(\tilde{u}'_n,\tilde{u}'_n)\\
+ & h^4\big( b_2 (h \tilde{J}_n)\frac{c^3_2}{3!}-\varphi_4 (h \tilde{J}_n) \big)\big( \tilde{g}^{(3)}_n(\tilde{u}_n)(\tilde{u}'_n, \tilde{u}'_n, \tilde{u}'_n)+3 \tilde{g}''_n(\tilde{u}_n)(\tilde{u}'_n, \tilde{u}''_n) \big) +\mathcal{O}(h^5).
\end{aligned}
\end{equation}
Here
\begin{equation} \label{eq4.6}
\tilde{J}_{n} = A+ g'(\tilde{u}_n), \q \tilde{g}_n(u)=F(u)-\tilde{J}_{n} u=g(u)-g'(\tilde{u}_n ) u.
\end{equation}
Note that since $\tilde{g}'_n (\tilde{u}_n)=0$ and $\tilde{g}^{(k)}_n (u)=g^{(k)}(u)$ ($k\geq 2$), one can actually replace \eqref{eq4.5} by  
\begin{equation} \label{eq4.7}
\begin{aligned}
\tilde{e}_{n+1}=& h^3 \big( b_2 (h \tilde{J}_n)\frac{c^2_2}{2!}-\varphi_3 (h \tilde{J}_n) \big) g''(\tilde{u}_n)(\tilde{u}'_n,\tilde{u}'_n)\\
+ & h^4\big( b_2 (h \tilde{J}_n)\frac{c^3_2}{3!}-\varphi_4 (h \tilde{J}_n) \big)\big( g^{(3)}(\tilde{u}_n)(\tilde{u}'_n, \tilde{u}'_n, \tilde{u}'_n)+3 g''(\tilde{u}_n)(\tilde{u}'_n, \tilde{u}''_n) \big) +\mathcal{O}(h^5).
\end{aligned}
\end{equation}
Requiring  $\tilde{e}_{n+1}=\mathcal{O}(h^{5})$ retrieves the stiff order conditions for methods of order 4 (see \cite{LO14a}), which are written for 2-stage methods as 
\[
b_2 (Z)\frac{c^2_2}{2!}=\varphi_3 (Z), \q b_2 (Z)\frac{c^3_2}{3!}=\varphi_4 (Z)
\]
with  $Z$ denotes an arbitrary square matrix. However, this is impossible due to the fact that matrix functions $\varphi_3 (Z), \varphi_4 (Z)$ are linearly independent.
We thus follow the similar remedy as presented in \cite[Sec.4.2]{LO14a} in order to relax the stiff order conditions. First, one realizes that there exist bounded operators $\widehat{b}_2 (h \tilde{J}_n)$ and $\widehat{\varphi}_4 (h \tilde{J}_n)$ such that
\begin{equation} \label{eq4.8}
b_2 (h \tilde{J}_n)\frac{c^3_2}{3!}-\varphi_4 (h \tilde{J}_n)=\big( b_2 (0)\frac{c^3_2}{3!}-\varphi_4 (0)\big)+h \Big(\widehat{b}_2 (h \tilde{J}_n)\frac{c^3_2}{3!}-\widehat{\varphi}_4 (h \tilde{J}_n) \Big) \tilde{J}_n.
\end{equation}
This is due to the recurrence relation \eqref{eq1.6} for $\varphi_k (z)$ and the fact that $b_2 (h \tilde{J}_n)$ is chosen as linear combinations of $\varphi_k (h \tilde{J}_n)$.
Inserting \eqref{eq4.8} into \eqref{eq4.7} gives 
\begin{equation} \label{eq4.9}
\begin{aligned}
\tilde{e}_{n+1}&=h^3 \big( b_2 (h \tilde{J}_n)\frac{c^2_2}{2!}-\varphi_3 (h \tilde{J}_n) \big) g''(\tilde{u}_n)(\tilde{u}'_n,\tilde{u}'_n) \\
&+ h^4\big( b_2 (0)\frac{c^3_2}{3!}-\varphi_4 (0)\big) \big( g^{(3)}(\tilde{u}_n)(\tilde{u}'_n, \tilde{u}'_n, \tilde{u}'_n)+3 g''(\tilde{u}_n)(\tilde{u}'_n, \tilde{u}''_n) \big)\\
&+ h^5 \Big(\widehat{b}_2 (h \tilde{J}_n)\frac{c^3_2}{3!}-\widehat{\varphi}_4 (h \tilde{J}_n) \Big) \tilde{J}_n \big( g^{(3)}(\tilde{u}_n)(\tilde{u}'_n, \tilde{u}'_n, \tilde{u}'_n)+3 g''(\tilde{u}_n)(\tilde{u}'_n, \tilde{u}''_n) \big)  + \mathcal{O}(h^5).
\end{aligned}
\end{equation}
This local error expansion brings us to the following result concerning the relaxing stiff order conditions for 2-stage methods of order 4.
\begin{lemma} \label{lm4.1}
Under Assumptions 2 and 3 and further assume that the operator $A$ and the nonlinearity $g(u)$ in \eqref{eq4.1}  are such that 
\begin{equation} \label{eq4.10}
A \big( g^{(3)}(\tilde{u}_n)(\tilde{u}'_n, \tilde{u}'_n, \tilde{u}'_n)+3 g''(\tilde{u}_n)(\tilde{u}'_n, \tilde{u}''_n) \big)
\end{equation}
is uniformly bounded on $X$, a 2-stage explicit exponential Rosenbrock method \eqref{eq2.1} has order of consistency five, i.e. $\tilde{e}_{n+1}=\mathcal{O}(h^{5})$ if the following order conditions are fulfilled 
\begin{subequations} \label{eq4.11}
\begin{align}
 b_2 (Z)c^2_2&=2\varphi_3 (Z), \\
b_2 (0)c^3_2 &=6\varphi_4 (0) 
\end{align}
\end{subequations}
with $Z$ denotes an arbitrary square matrix. Moreover, the remainder term of $\tilde{e}_{n+1}$, which is hidden behind the Landau notation $\mathcal{O}(\cdot)$, is bounded by $Ch^5$ with a constant $C$ that only depends on values that are uniformly bounded by the assumptions made, i.e., is independent of $n$ and $h$. 
 \end{lemma}
 \begin{proof}
 It follows at once from Assumption~3 and the additional regularity assumption \eqref{eq4.10} that $\tilde{J}_n\big( g^{(3)}(\tilde{u}_n)(\tilde{u}'_n, \tilde{u}'_n, \tilde{u}'_n)+3 g''(\tilde{u}_n)(\tilde{u}'_n, \tilde{u}''_n) \big)$ is also uniformly bounded on $X$ (since $\tilde{J}_{n} = A+ g'(\tilde{u}_n)$). 
In view of \eqref{eq4.9}, the conclusion of Lemma~\ref{lm4.1} is thus verified by using the Assumptions 2--3 and the given order conditions in \eqref{eq4.11}.
 \end{proof}
 \emph{Remark~2.} The additional smoothness condition \eqref{eq4.10} is often fulfilled for many semilinear parabolic PDEs such as reaction-diffusion equations, the Allen-Cahn equation and the Chafee-Infante problem \cite[Chap.~5]{H81}, where the operator $A$ is the strongly second-order elliptic differential operator (e.g. the Laplacian or the gradient). In particular, for such problems, one can show that Assumption~3 implies \eqref{eq4.10}. For more details, we refer to \cite[Example~4.1]{LO14a}. 
\subsection{Derivation of a fourth-order 2-stage stiffly accurate scheme}
Solving the order conditions in \eqref{eq4.11} gives $b_2 (0)=\dfrac{2\varphi_3(0)}{c^2_2}=\dfrac{6\varphi_4(0)}{c^3_2}$ which implies $c_2=\dfrac{3\varphi_4(0)}{\varphi_3(0)}=\dfrac{3}{4}$ and thus $b_2 (Z)=\dfrac{32}{9}\varphi_3 (Z)$. This is the unique solution of \eqref{eq4.11}. Inserting this result into \eqref{eq2.1} we obtain the following 2-stage scheme which will be called $\mathtt{exprb42}$:
\begin{subequations} \label{eq4.12}
\begin{align}
 U_{n2}&= u_n + \frac{3}{4} h \varphi _{1} ( \frac{3}{4} h J_n)F(u_n), \\
u_{n+1}& = u_n + h \varphi _{1} ( h J_n)F(u_n) + h \frac{32}{9}\varphi_3 (h J_n) (g_n (U_{n2})- g_n(u_n )).  
\end{align}
\end{subequations}
The convergence of this scheme will be stated in the next section.
 \subsection{Stability and convergence result}\label{subsection4.13}
It is shown in \cite[Sec. 3.3]{HOS09} that the following stability bound
\begin{equation} \label{eq4.13}
\Bigl\|\prod_{j=0}^{n-k} \ee^{h J_{n-j}} \Bigr\|_{X\leftarrow X}\leq C_\text{\rm S},\qquad t_0\le t_k \le t_n\le T
\end{equation}
is the key to show the convergence of exponential Rosenbrock methods \eqref{eq1.4}. The good thing here is that the constant $C_\text{\rm S}$ in \eqref{eq4.13} is uniform in $k$ and $n$ despite the fact that $J_n$ varies from step to step.

With the help of \eqref{eq4.13}, one can prove that $\mathtt{exprb42}$ converges with global order 4 by using the same techniques presented in the recent work \cite[Sec. 4]{LO14a} (presenting the convergence results for methods of orders up to 5).  
For convenience for the reader, below we recall some of the important results which can be applied directly to our case ($s=2, \ c_2=\frac{3}{4}, \ b_2 (h J_n)=\frac{32}{9}\varphi_3 (h J_n) $). However, we will omit other details of their proofs. 

Let $ e_{n+1} = u_{n+1} - u(t_{n+1})=u_{n+1} - \tilde{u}_{n+1}$ denote the global error of the scheme \eqref{eq4.12}. One can show that it satisfies 
\begin{equation} \label{eq4.14}
e_{n+1}=\ee^{h J_{n}}e_n + h P_n +\tilde{e}_{n+1}, \q e_0=0
\end{equation}
with 
\begin{equation} \label{eq4.15}
\begin{aligned}
P_n &= \varphi _{1} ( h J_{n})\big(g_n (u_n)-g_n (\tilde{u}_n)\big)+\big( \varphi _{1} ( h J_{n})-\varphi _{1} ( h \tilde{J}_{n}) \big) F(\tilde{u}_n)\\
&+ b_2 (h J_n) (g_n(U_{n2})-g_n(u_n) )  - b_2 (h \tilde{J}_n)(\tilde{g}_n(\overbar{U}_{n2})-\tilde{g}_n(\tilde{u}_n) ). 
\end{aligned}
\end{equation}
Under Assumptions 2 and 3, the following estimate 
 \begin{equation} \label{eq4.16}
 \|P_n \|\leq C\|e_n\| +C\|e_n\|^2 +Ch^6
 \end{equation}
 holds true as a direct result of Lemma 4.5 in \cite{LO14a}.
 
We are now at the final stage of formulating our convergence result.
\begin{theorem}\label{th4.1}
Let the initial value problem \eqref{eq4.1} satisfy  the Assumptions of Lemma~\ref{lm4.1}. Then, the numerical solution $u_n$ of the 2-stage explicit exponential Rosenbrock method \ $\mathtt{exprb42}$ satisfies the error bound
\begin{equation}\label{eq4.17}
\| u_n -u(t_n)\|\leq C h^4
\end{equation}
uniformly on \ $t_0 \leq  t_n =t_0+nh \leq  T$ with a constant $C$ that depends on $T-t_0$, but is independent of $n$ and $h$.
\end{theorem}
\begin{proof}
Solving the recursion \eqref{eq4.14} gives
\begin{equation} \label{eq4.18}
e_{n}=h\sum_{k=0}^{n-1}  \prod_{j=1}^{n-k-1} \ee^{h J_{n-j}}  \Bigl(P_k + \frac{1}{h}\tilde{e}_{k+1} \Bigr).
\end{equation}
The result of Lemma~\ref{lm4.1} shows that the local error of $\mathtt{exprb42}$ satisfies $\tilde{e}_{k+1} =\mathcal{O}(h^5)$. 
Next, we use the stability estimate (\ref{eq4.13}) and the bound (\ref{eq4.16}) to get
\begin{equation} \label{eq4.19}
\| e_{n}\| \leq C \sum_{k=0}^{n-1} h \big( \| e_{k}\|  +\| e_{k}\|^2   + h^{4}\big).
\end{equation}
The desired bound \eqref{eq4.17} follows by an application of a discrete Gronwall lemma (see \cite{Em05}) to (\ref{eq4.19}).
\end{proof}

\section{Adaptive time-stepping schemes}\label{sc5}
It should be mentioned that the newly constructed schemes  $\mathtt{exprb42}$ (see \eqref{eq4.12}) and $\mathtt{exprb42N}$ (see \eqref{eq2.20}) can also be implemented with variable stepsizes.  
Indeed one can use the standard way as employed in \cite{HOS09,LO14a} (for other exponential Rosenbrock schemes) that is to consider \eqref{eq2.1} together with an embedded scheme of lower order
\begin{equation} \label{eq5.1}
\hat{u}_{n+1}= u_n + h \varphi _{1} ( h J_n)F(u_n) + h \widehat{b}_{2}(h J_n) (g_n (U_{n2})- g_n(u_n ))
\end{equation}
which uses the same internal stage $U_{n2}$. It is clear that a 2-stage method of order 3 requires $b_2(Z)=\dfrac{2\varphi_3(Z)}{c^2_2}$ for any node $c_2 \ne 0$. For $c_2=3/4$ that is uniquely determined (so is $U_{n2}$) by the construction of the two new schemes, it has been shown that such a 2-stage method can even attain order 4. This implies that it is impossible to embed  $\mathtt{exprb42}$ with a 2-stage method of order 3. The fact that $\mathtt{exprb42N}$ has the same $U_{n2}$ as  $\mathtt{exprb42}$,  we thus consider to embed both of them with a second-order error estimate, which is the exponential Rosenbrock-Euler method (so $\widehat{b}_{2}(h J_n)=0$).
For later use in our numerical experiments, we display $\mathtt{exprb42}$  and $\mathtt{exprb42N}$  (for variable stepsizes implementation) in reduced Butcher tableau (see \cite[Sect.2]{LO16}) as follows

\hspace{2cm} $\mathtt{exprb42N}$: \hspace{3.8cm} $\mathtt{exprb42}$:

\footnotesize
\begin{displaymath}
\renewcommand{\arraystretch}{1.3}
\begin{array}{c|c}
\dfrac{3}{4}&\\ [1pt]
\hline
& -\cfrac{8}{27}\varphi _1 +\cfrac{48}{27}\varphi _2 \\[1pt]
& 0
\end{array}
\q \q  \q  \q  \q \q  \q 
\begin{array}{c|c}
\dfrac{3}{4}&\\ [1pt]
\hline
& \cfrac{32}{9}\varphi_3 \\[1pt]
& 0
\end{array}.
\end{displaymath}
\normalsize
\section{Numerical experiments}\label{sc6}
In this section we verify our convergence results and demonstrate the efficiency of the new integrators $\mathtt{exprb42N}$  and $\mathtt{exprb42}$. To this aim, we carry out numerical experiments on a set of test problems (see below). First, we discuss the implementation of the new integrators.
\subsection{Implementation and test problems}
\subsubsection{Implementation}
The implementation of exponential integrators (in particular, the new integrators $\mathtt{exprb42N}$ and $\mathtt{exprb42}$) requires computing the action of matrix functions $\varphi_k (h J_n)$ on vectors $v_k$. With the recent developments of numerical linear algebra in computing matrix functions (see, for example \cite{AH11,NW12,CKOS14}), this can be done efficiently. In order to take advantages of computing a linear combination of terms like $\sum_{k= 1}^{p}\varphi_k (h J_n) v_k$ (by one single evaluation) and computational time, we use here the adaptive Krylov technique proposed in \cite{NW12, TLP12}. For variable step sizes implementation, the error estimate ${\tt err}=u_{n+1}- \hat{u}_{n+1}$ (see \cite[Chapter IV.8]{HW96}) will be used to control time steps. 
All the simulations are run in MATLAB.  

Next, we give a list of test problems including both nonstiff and stiff differential equations that fit in the framework.
\subsubsection{Nonstiff problems}
\begin{example}\label{ex1}\rm
Consider an example from Astronomy-the restricted three body problem (see \cite{HNW93,PWW2000}): 
\begin{equation} \label{example1}
y''_1 =-\frac{y_1}{(y^2_1+ y^2_2)^{3/2}}, \q
y''_2 =-\frac{y_2}{(y^2_1+ y^2_2)^{3/2}}, \q t \in [0, 10].
\end{equation}
The equation of motion above can be written as a system of first-order differential equations as $y'_1=y_3, \ y'_2=y_4, \ y'_3=- y_{1}/r^{3}, \ y'_4=-y_{2}/r^{3}$ with $r=\sqrt{y^2_1+ y^2_2}$.
For this simple case (which can be considered as a two-body orbit problem) the exact solution is known, that is $y(t)=[\cos (t), \sin (t), - \sin (t),  \cos (t)]$.
\end{example}
\begin{example}\label{ex2}\rm
Consider the van der Pol equation (see \cite{HNW93}): 
\begin{equation} \label{example2}
\begin{aligned}
y'_1& =y_2, \\
y'_2 &=(1-y^2_1)y_2-y_1, \q t \in [0, 2]\\
y_1&(0)=2, \ y_2(0)=0.
\end{aligned}
\end{equation}
Since the exact solution of \eqref{example2} is unknown, we compute its reference solution by using a nonstiff solver such as \emph{ode45} with ATOL=RTOL=$10^{-14}$.
\end{example}
\subsubsection{Stiff problems}
\begin{example}\label{ex3}\rm
Consider the one-dimensional semilinear parabolic problem (see \cite{HO05})
\begin{equation} \label{example3}
\frac{\partial u}{\partial t}- \frac{\partial^2 u}{\partial x^2}  =\frac{1}{1+u^2}+\Phi (x,t)
\end{equation}
for $u=u(x,t)$ on the unit interval $[0,1]$ and  $t \in [0,1]$, subject to homogeneous Dirichlet boundary conditions. The source function $\Phi$ is chosen in such a way that the exact solution of the problem is $u(x,t)=x(1-x)\ee^t$.
\end{example}
In order to solve \eqref{example3} numerically, the first step is to discretize it in space by standard finite differences with $M=199$ (inner) grid points. This yields a very stiff system of the form \eqref{eq4.1} (with $\|A\|_{\infty}=1.5999e+05$). Then we use our new integrators to integrate this ODE system in time with constant step sizes. Let  $U^{n,h}_i \approx u(x_i, t_n)$ denote the numerical solution at $t_n = nh$ and grid point $x_i = \frac{i}{200}$ and let $\bar u_i$ denote a reference solution of the spatially discrete problem at time $t=1$ and grid point $x_i$, computed with sufficiently small time steps. Note that for this example, since we know the exact solution, one can take $\bar u_i=x_i (1-x_i)$. The time integration errors $U^{N,1/N}-\bar u$ are measured in the maximum norm $\max_{1\le i\le 199}|U^{N,1/N}_i-\bar u_i|$. 
 
 \begin{example}\label{ex4}\rm
Consider the two-dimensional advection-diffusion-reaction equation (see, for example,  \cite{HOS09, LO14a})
\begin{equation} \label{example4}
\frac{\partial u}{\partial t}=0.01 \Delta u +10 \nabla u  + 100 u \big(u-\tfrac{1}{2} \big) (1-u)
\end{equation}
for $u=u(x,y,t)$ on the unit square $ \Omega =[0,1]^2$  with the initial value
\begin{equation*}
u(x,y,0)=0.3+ 256\big(x(1-x)y(1-y)\big)^2,
\end{equation*}
subject to homogeneous Neumann boundary conditions.
Here $\Delta$ and $\nabla$ denote the Laplacian and the gradient vector field in two dimensions, respectively. Discretizing \eqref{example4} in space by standard finite differences using $101$ grid points in each direction with meshwidth $\dd x=\dd y=1/100$ yields a mildly stiff system of the form \eqref{eq4.1} (with $\|A\|_{\infty}=2.4e+03$). For the time integration of this resulting system of ODEs, we use our new integrators. Since the exact solution of \eqref{example4} is unknown,  a reliable reference solution is computed by using sufficiently small time steps (one can also use the stiff solver $\mathtt{ode15s}$ with ATOL=RTOL=$10^{-14}$). As done for Example~\ref{ex3}, the time integration errors are  measured in a discrete maximum norm at the final time $T=0.08$.
\end{example}
\subsection{Accuracy verification and performance comparison}
The purpose of giving the two nonstiff problems in Examples~\ref{ex1} whose exact solution is known and in Examples~\ref{ex2}, whose exact solution is unknown, is just to verify the order 4 of the integrator $\mathtt{exprb42N}$ (satisfying the classical order conditions). However, we also display the order plots of the stiff integrator $\mathtt{exprb42}$ in Fig.~\ref{fig1} (in a double-logarithmic diagram). In this experiment, we use constant step sizes that correspond to the number of time steps that are $N=64$, 128, 256, 512. The diagrams clearly shows a perfect agreement with Theorem~\ref{theorem2.1}. It is observed that with the same number of time steps $\mathtt{exprb42}$ even gets a bit more accuracy than $\mathtt{exprb42N}$ for Example~\ref{example1}. For Example~\ref{example2} both integrators give almost identical results.
\begin{figure}[H]
\centering
\begin{tabular}{cc}
\epsfig{file=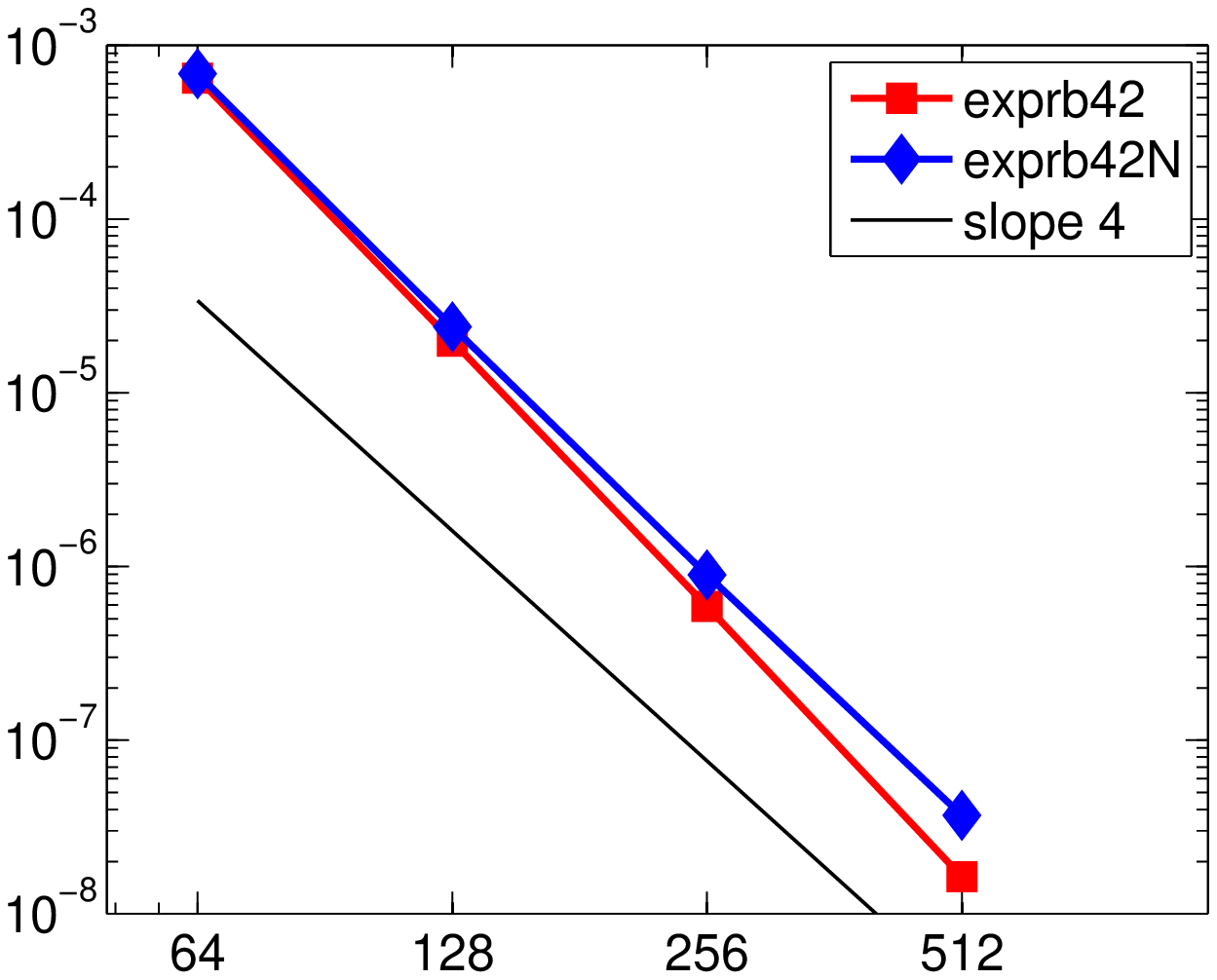,width=0.47\linewidth,clip=}

\epsfig{file=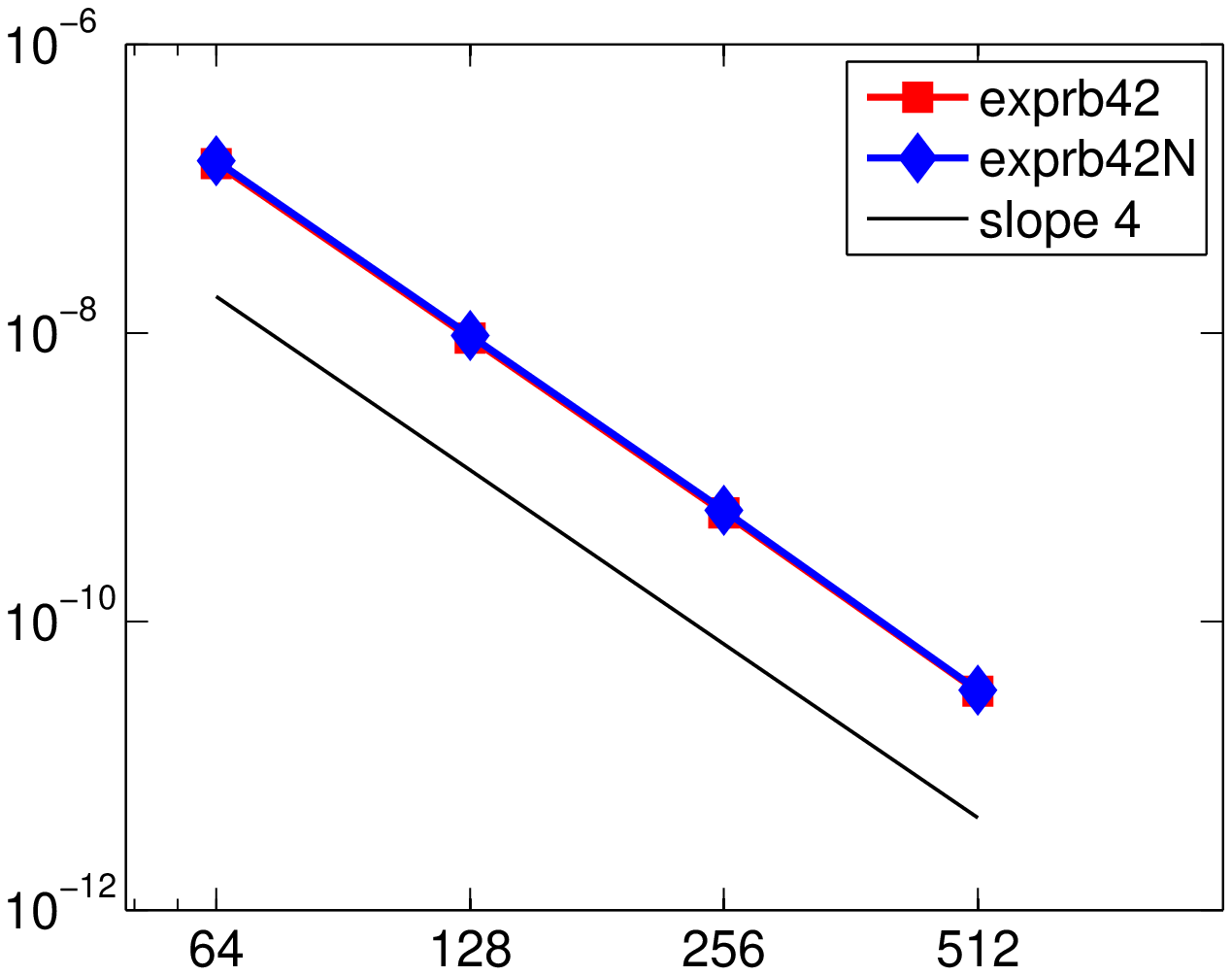,width=0.47\linewidth,clip=}  
\end{tabular}
\vspace{-6mm}
\caption{\label{fig1} Order plots of $\mathtt{exprb42N}$ and $\mathtt{exprb42}$ when applied to Example~\ref{example1} (left) and Example~\ref{example2} (right). The errors at time $t=10$ (left) and $t=2$ (right) are plotted as functions of the number of time steps $N$. For comparison, a straight line with slope 4 is added.}
\end{figure}
For the two stiff problems in Examples~\ref{ex3} and \ref{ex4}, we also check the sharpness of the error bound given in Theorem~\ref{th4.1}. Since the problems are stiff, it is interesting to compare the two new 2-stage fourth-order explicit integrators with the 2-stage Gauss--Legendre scheme-the only existing class of 2-stage fourth-order method (see \cite{HW96}). We will call it as $\mathtt{Gauss42}$ for the rest of the paper. Since $\mathtt{Gauss42}$ is an implicit Runge--Kutta method, we use the simplified Newton iterations as suggested in \cite{HW96} for its implementation. 
\begin{figure}[H]
\centering
\begin{tabular}{cc}
\epsfig{file=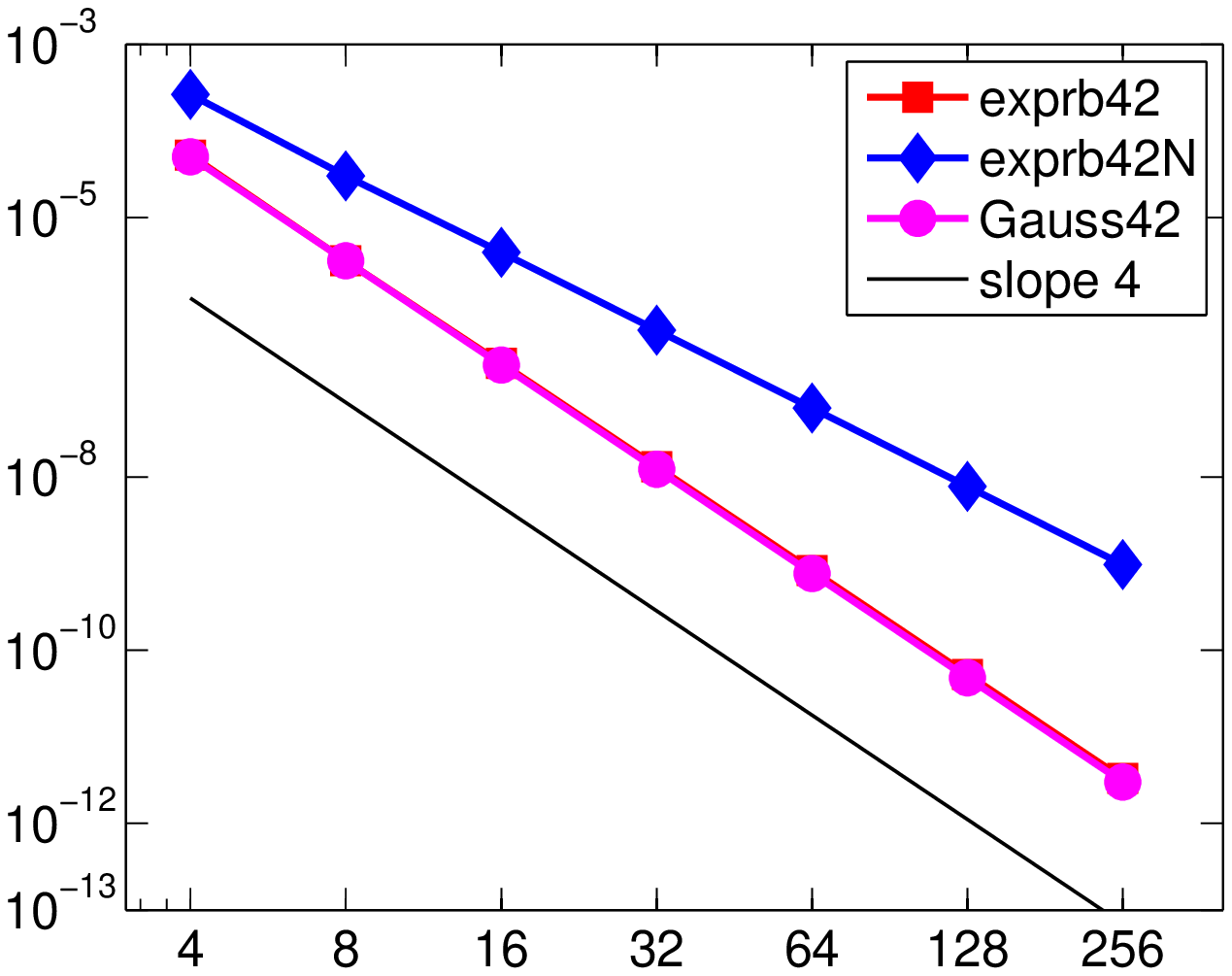,width=0.47\linewidth,clip=}

\epsfig{file=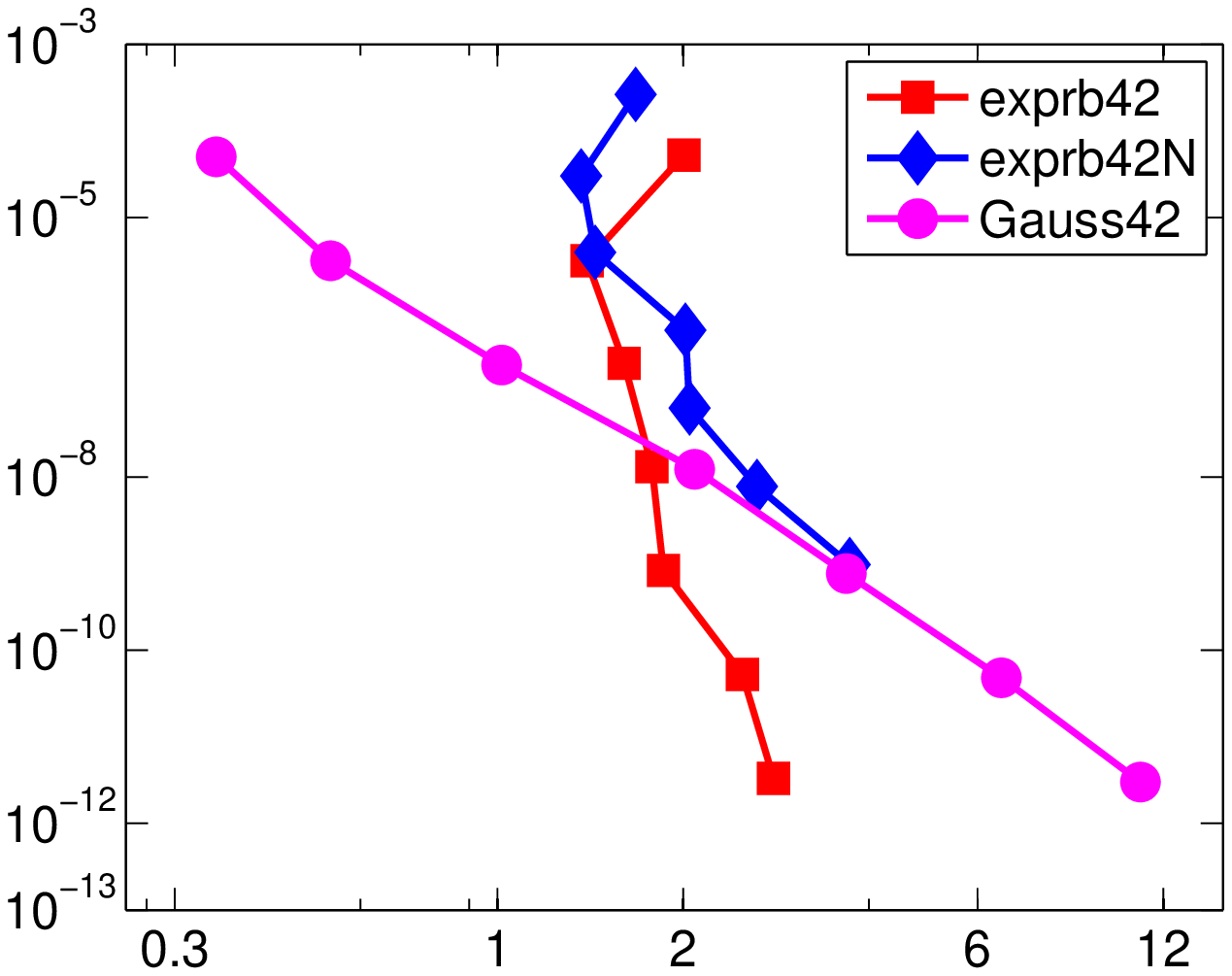,width=0.47\linewidth,clip=}  
\end{tabular}
\vspace{-6mm}
\caption{\label{fig2} Order plots (left) and total CPU times (right) of $\mathtt{exprb42N}$, $\mathtt{exprb42}$, and $\mathtt{Gauss42}$ when applied to Example~\ref{example3}. The errors at time $t=1$ are plotted as functions of the number of time steps $N=4,8,16,32,64,128,256$ (left) and the total CPU time in second (right). For comparison, a straight line with slope 4 is added.}
\end{figure}
As seen from the left diagram in Fig.~\ref{fig2}, while $\mathtt{exprb42N}$ suffers from order reduction when applied to the very stiff problem in Example~\ref{example3} , the two stiff solvers $\mathtt{exprb42}$ and $\mathtt{Gauss42}$ achieve perfectly order 4 and give almost identical global errors for a given number of time steps. In the right diagram we plot the total CPU time versus global error. It turns out that $\mathtt{exprb42}$ is the fastest one  for more stringent global error tolerances (much faster than $\mathtt{Gauss42}$). Moreover, as the number of time steps increases, the CPU time of $\mathtt{exprb42}$ does not increase much while $\mathtt{Gauss42}$ increases CPU time with rate in an approximately linear manner. This can be explained as $\mathtt{Gauss42}$ requires solving a nonlinear system of equations at every step.

In Fig.~\ref{fig3}, we again use constant step sizes (corresponding to the number of time steps $N=32,64,128,256, 512$) to verify the achieved orders of the two new integrators when applied to Example~\ref{example4}. Along with $\mathtt{Gauss42}$ we also added the best fourth-order exponential Rosenbrock scheme, $\mathtt{pexprb43}$, which requires 3 stages (see \cite{LO16}) to this comparison. The left precision diagram clearly confirms that all integrators are indeed of order 4, meaning that $\mathtt{exprb42N}$ does not suffer from order reduction for this mildly stiff problem.  In addition, we see that the three fourth-order exponential integrators even offer more accuracy than $\mathtt{Gauss42}$ for a given number of time steps.  For this two-dimensional problem, the right precision diagram indicates a huge computational saving of the three exponential schemes over $\mathtt{Gauss42}$. This is again due to the implicitness of scheme $\mathtt{Gauss42}$. Furthermore, it is observed that both $\mathtt{exprb42N}$ and $\mathtt{exprb42}$ are a bit faster than $\mathtt{pexprb43}$.
\begin{figure}[H]
\centering
\begin{tabular}{cc}
\epsfig{file=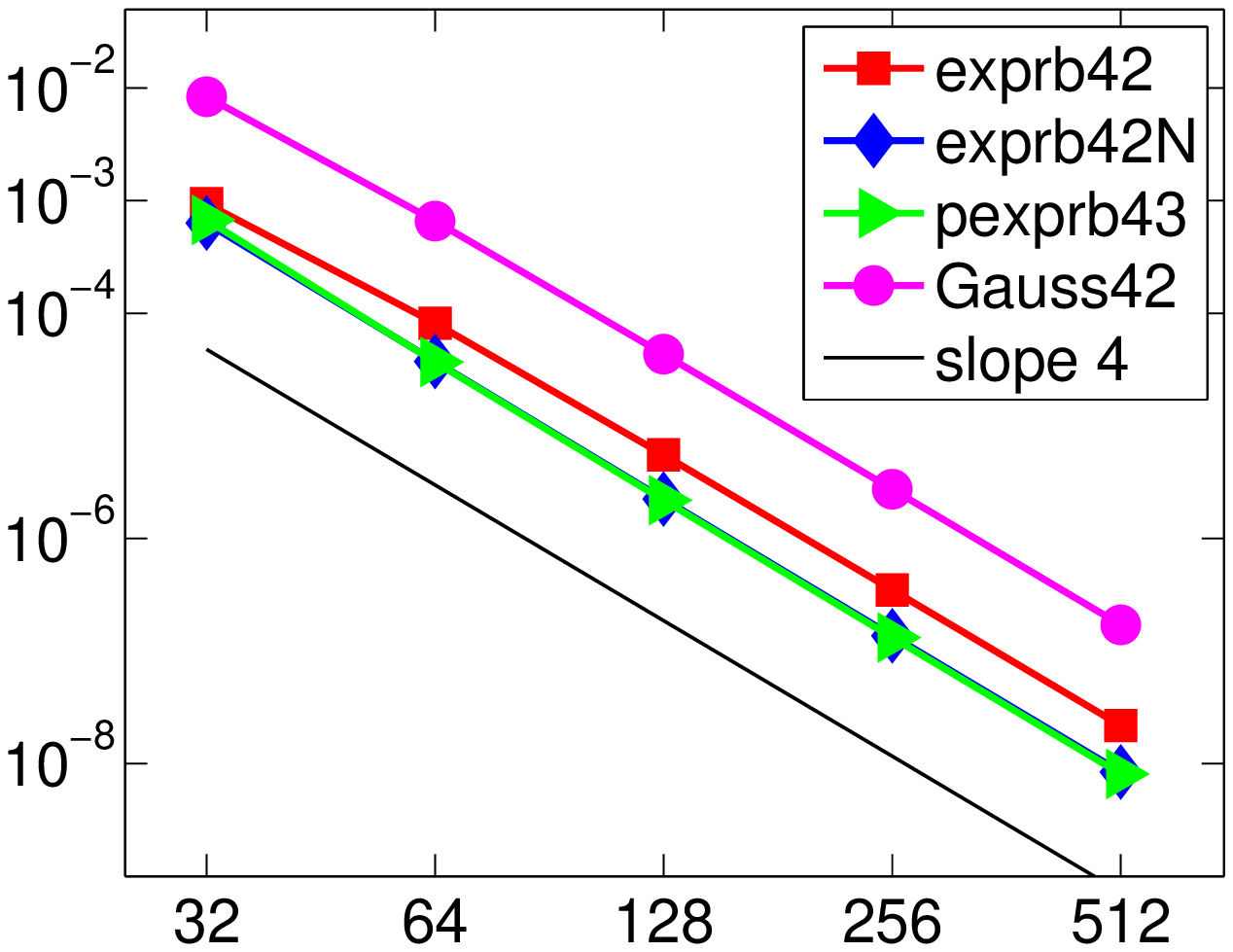,width=0.47\linewidth,clip=}
\epsfig{file=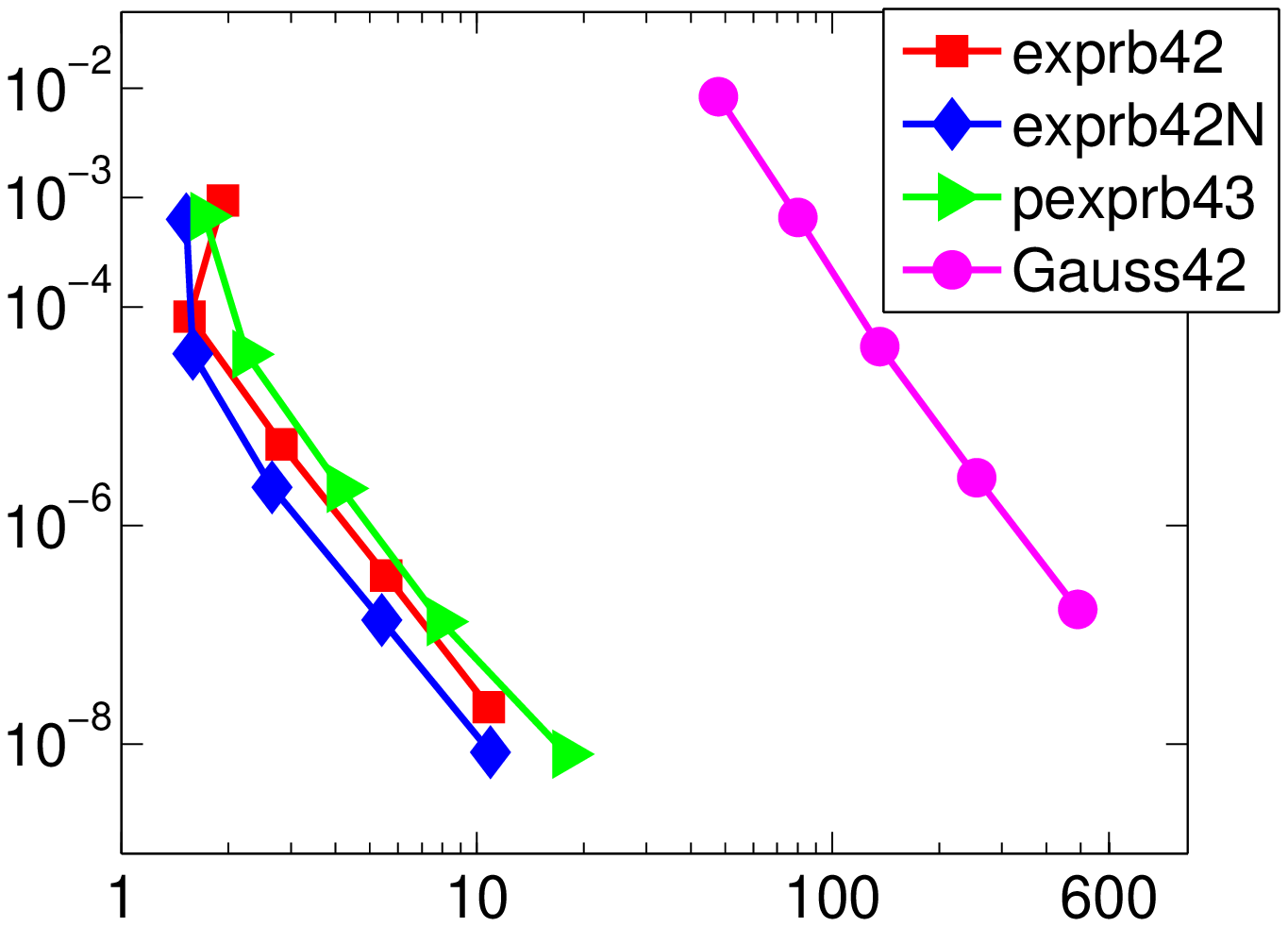,width=0.47\linewidth,clip=}  
\end{tabular}
\vspace{-6mm}
\caption{\label{fig3} Order plots (left) and total CPU times (right) of $\mathtt{exprb42N}$, $\mathtt{exprb42}$, $\mathtt{pexprb43}$, and $\mathtt{Gauss42}$ when applied to Example~\ref{example4}. The errors at time $t=0.08$ are plotted as functions of the number of time steps $N=32,64,128,256, 512$ (left) and the total CPU time in second (right). For comparison, a straight line with slope 4 is added.}
\end{figure}
Next, we implement the new integrators using variable step sizes codes. In Fig.~\ref{fig4}, using the same tolerances ATOL $=$ RTOL $=10^{-4}, 10^{-4.5},\ldots,10^{-6}$ we plot the achieved accuracy as a function of the required  number of time steps. The results are compared again with those of  $\mathtt{pexprb43}$.
\begin{figure}[H]
\centering
\begin{tabular}{cc}
\epsfig{file=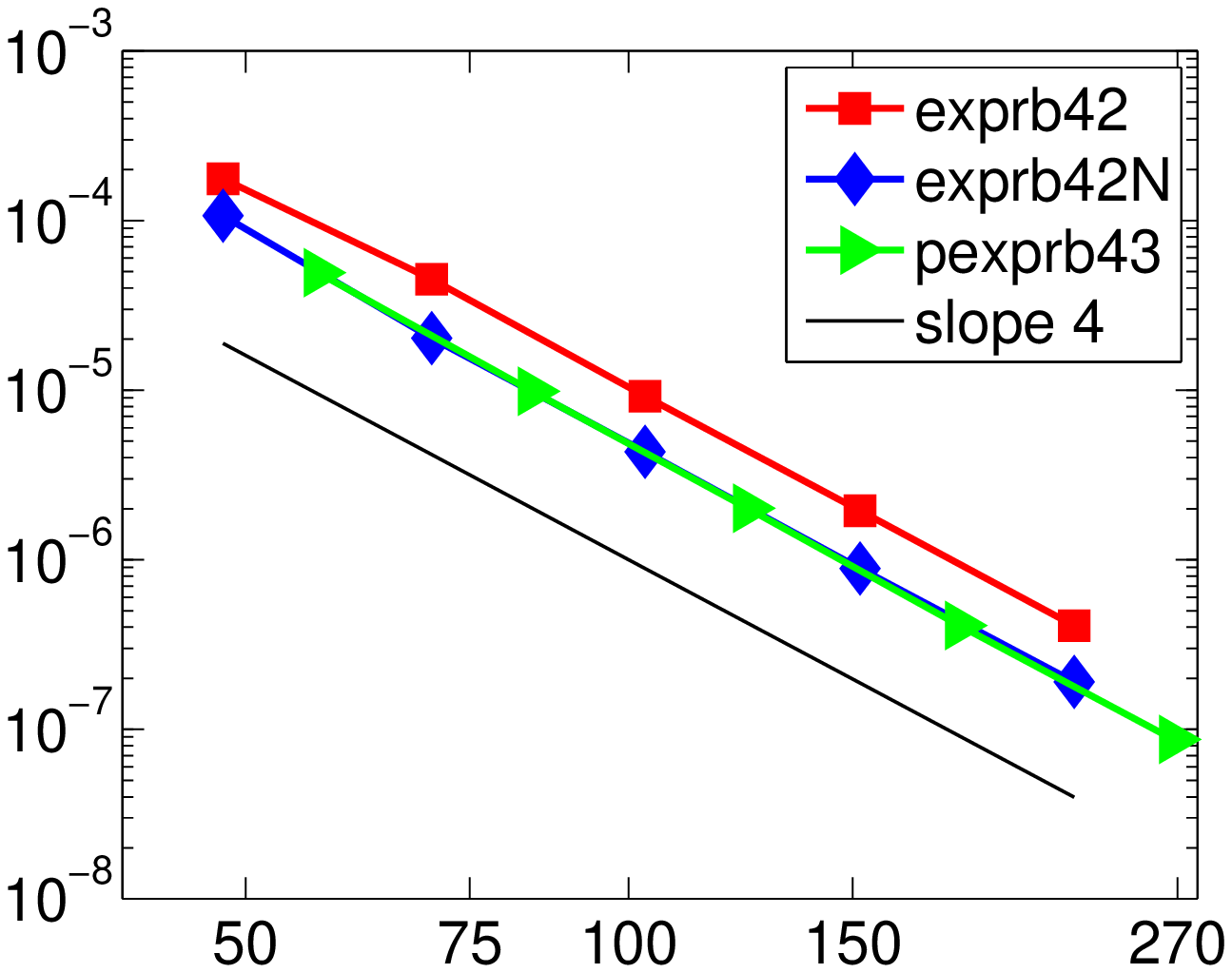,width=0.47\linewidth,clip=}
\epsfig{file=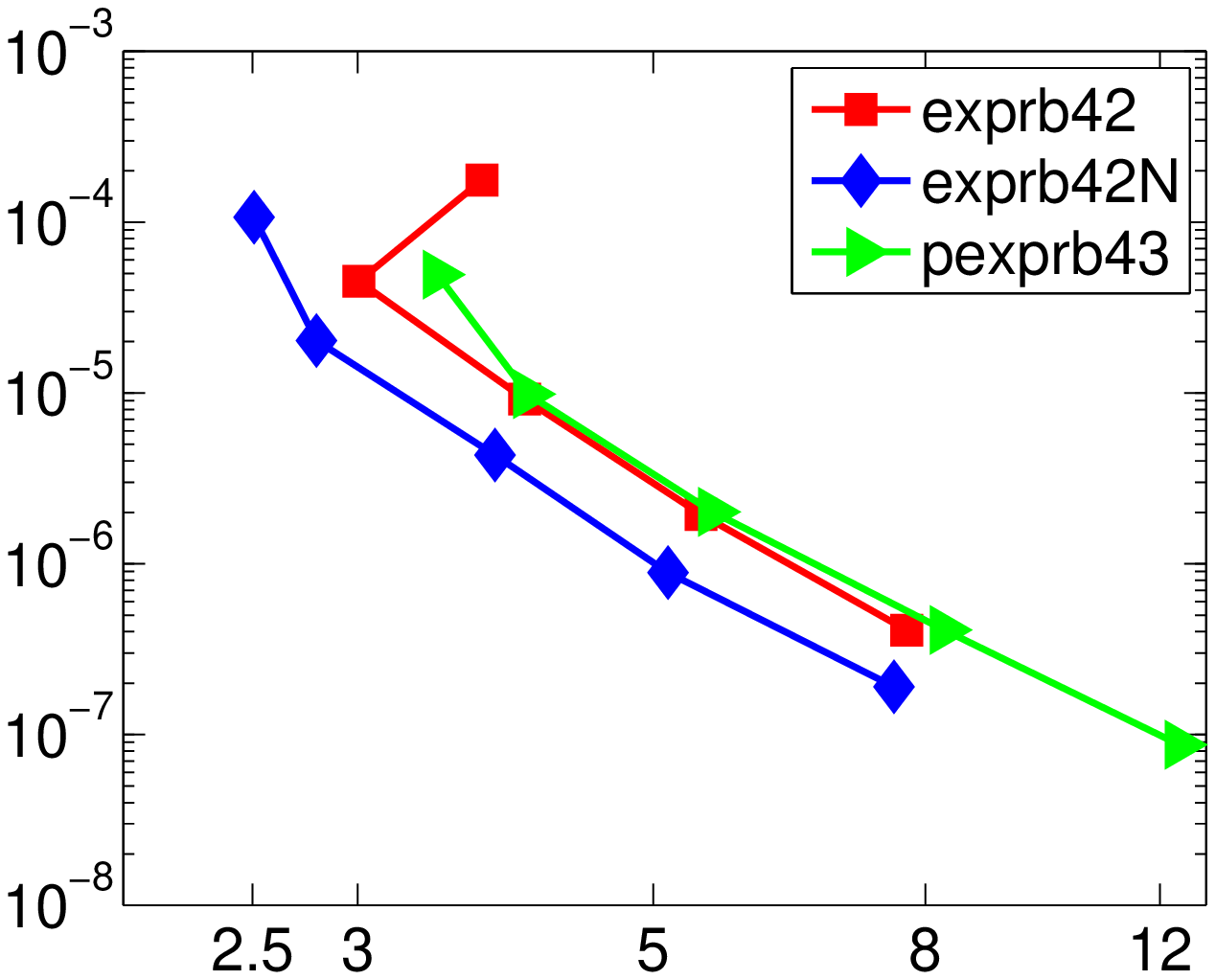,width=0.47\linewidth,clip=}  
\end{tabular}
\vspace{-6mm}
\caption{\label{fig4} Number of time steps versus accuracy (left) and the total CPU time versus accuracy (right) for the advection-diffusion-reaction Example~\ref{example4} for $t=0.08$. The errors are measured in a discrete maximum norm.}
\end{figure}
 The precision diagrams in Fig.~\ref{fig4} indicates that $\mathtt{pexprb43}$ gets a bit more accuracy but takes more number of time steps as well as requires more CPU time than $\mathtt{exprb42N}$, $\mathtt{exprb42}$. This observation is fairly comparable with the experiments using constant step sizes in Fig.~\ref{fig3}. 
\begin{figure}[H]
\begin{center}
\includegraphics[scale=0.45]{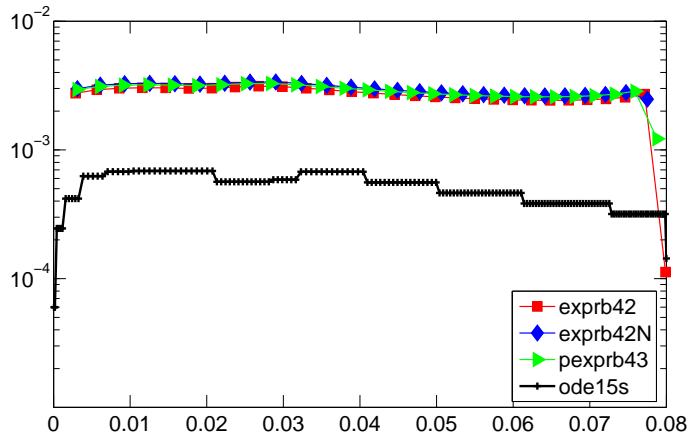}
\end{center}
\vspace{-9mm}
\caption{\label{fig5} Time versus step sizes for the advection-diffusion-reaction Example~\ref{example4} for an accuracy of about 0.001 at $t=0.08$.}
\end{figure}
Finally, we fix a final accuracy of about 0.001 at $t=0.08$ (by choosing appropriate tolerances) for integrators $\mathtt{exprb42N}$, $\mathtt{exprb42}$, $\mathtt{pexprb43}$ as well as a well-established and widely-use code-the stiff solver $\mathtt{ode15s}$ in order to compare their chosen step sizes. As seen from
Fig.~\ref{fig5}, the new integrators $\mathtt{exprb42N}$ and $\mathtt{exprb42}$ use about the same steps as $\mathtt{pexprb43}$ (28-30 steps) and take much larger time steps compared to $\mathtt{ode15s}$ (165 steps). Overall, we conclude that both $\mathtt{exprb42N}$ and $\mathtt{exprb42}$ perform quite well and they certainly beat the implicit methods such as $\mathtt{Gauss42}$ and $\mathtt{ode15s}$ for the advection-diffusion-reaction Example~\ref{example4}.


\bibliographystyle{elsarticle-num}
\bibliography{references_twostageExpint4}

\end{document}